\documentclass[preprint,10pt]{elsarticle}

\usepackage{amsmath,amsfonts,amssymb}
\usepackage{color}
\usepackage[hidelinks]{hyperref}
\usepackage{graphicx}
\usepackage{bm}

\newproof{proof}{Proof}
\newtheorem{lemma}{Lemma}[section]
\newtheorem{theorem}{Theorem}[section]

\usepackage[a4paper, total={6.5 in, 9.25 in}]{geometry}

\newcommand{\vect}[1]{\boldsymbol{#1}}

\journal{Applied Mathematics Letters}

\begin{document}

\begin{frontmatter}



\title{Global stability in a negative chemotaxis system with\\ chemically induced lethality}


\author[ucm,icai]{Federico Herrero-Hervás\corref{cor1}}
\ead{fedher01@ucm.es}
\author[ucm]{Mihaela Negreanu}
\ead{negreanu@mat.ucm.es}

\cortext[cor1]{Corresponding author}

\address[ucm]{Departamento de Análisis Matemático y Matemática Aplicada, Instituto de Matemática Interdisciplinar, Universidad Complutense de Madrid, Madrid, Spain.}
\address[icai]{Departamento de Matemática Aplicada, ICAI, Universidad Pontificia de Comillas, Madrid, Spain.}
\begin{abstract}

In this paper, we investigate the long-time dynamics of a repulsive Keller-Segel chemotaxis system. The model features negative chemotaxis, logistic growth and a cell death term, accounting for a lethal chemorepellent that is self-produced by the cells and externally supplied. We prove that, for constant chemorepellent supplies, depending on their magnitude with respect to the logistic growth rate, solutions converge in $L^\infty$ norm toward extinction of the population, or equilibrate toward a nontrivial spatially homogeneous steady state.
\end{abstract}


\end{frontmatter}

\section{Introduction}\label{intro}
\textit{Chemotaxis} is a biological phenomenon by which a wide range of living systems are able to orient their movement in response to chemical stimuli. Different chemical agents, for instance nutrients, growth factors, recruitment signals, or proteins lead to attractive motion, known as \textit{positive chemotaxis} or \textit{chemoattraction}. On the contrary, other substances, such as antibiotics or toxins induce repulsive dynamics, which are referred to as \textit{negative chemotaxis} or \textit{chemorepulsion}. These substances, respectively named \textit{chemoattractants} or \textit{chemorepellents}, may in some cases be self-produced by individuals of the species in order to foster self-aggregation or inhibition (for a detailed description, see for instance \cite{chem-bio-rev})

In this paper, we study a model that has been recently proposed in \cite{FH1} to describe the interactions between a biological species and a self-produced lethal chemorepellent. The model assumes that cells undergo a diffusive process, as well as chemorepulsion, migrating towards areas of lower substance concentration. A logistic term is added, to account for the population growth, as well as a term for the substance-induced lethality. For the chemorepellent, aside from diffusion and self-production by the individuals, the model includes time degradation and an external supply by means of a known function $f$.
In this way, the dynamics are described by the following system of two coupled parabolic partial differential equations
\begin{equation}\label{1.1}
     \begin{cases}
    \displaystyle    \frac{\partial u}{\partial t} = D \Delta u + \chi \nabla \cdot ( u \nabla v) + ru(1-u) - uv , & \quad x \in \Omega, ~ t > 0, \\[1.5ex]
   \displaystyle \frac{\partial v}{\partial t} = \Delta v + a u - v + f(x,t), & \quad x \in \Omega, ~ t > 0, \\[1.5ex]
    \displaystyle \frac{\partial u}{\partial \nu} = \frac{\partial v}{\partial \nu} = 0,  & \quad x \in\partial \Omega, ~ t>0, \\[1.5ex]
    \displaystyle u(x,0) = u_0(x), \quad v(x,0) = v_0(x),  & \quad x \in \Omega,
    \end{cases}
\end{equation}
where $u$ represents the population density of the species, $v$ the concentration of the chemorepellent, and $\nu$ the outward unit normal vector to $\partial \Omega$. Parameters $D, ~\chi,~r>0$ respectively representing the diffusion coefficient of the bacteria, the chemotaxis sensitivity coefficient, and the logistic growth rate, while $a, f(x,t)\geq 0$ are the self-production rate of the substance and the external supply, respectively.

System \eqref{1.1} is an extension of the well-known Keller-Segel model for chemotaxis \cite{KS1, KS2}, to which an extensive body of research has been devoted over the years (see for instance the surveys \cite{Bellomo, Horstmann}). While the majority of the studies concern positive chemotaxis, other negative chemotaxis systems similar to \eqref{1.1} have recently been analyzed, as for instance \cite{Deng2023, CieslakFuestHajdukSierzega2024}, or a combination of attraction and repulsion toward two different chemicals \cite{Tao2013-rep, Ren2020, Frassu2021}.

For system \eqref{1.1}, the current results are limited on the one hand to the existence and uniqueness of solutions for $f \in L^\infty(\Omega \times (0,\infty)) \cap C^1(\hspace{0.05 cm}\overline{\Omega} \times (0, \infty))$ for any $r>0$ for bounded domains with dimension $n \leq 2$, or for large enough $r>0$ on convex bounded domains for arbitrary $n \in \mathbb{N}$.

Moreover, asymptotic periodicity properties of $u$ and $v$ have been studied if $f(x,t)$ has a persistent periodic behavior in \cite{FH3}, and a linearized steady state analysis for constant supplies, i.e. $f(x,t) \equiv f$ in \cite{FH1}. In such a case, system \eqref{1.1} admits two spatially homogeneous steady states, given by 
\begin{equation}\label{estados}
    (0,f), \quad (u_*, v_*) := \displaystyle \left(\frac{r-f}{r+a}, \frac{r(f+a)}{r+a} \right),
\end{equation}
where the latter is only biologically meaningful if $r > f$. Particularly, the results in \cite{FH1} highlight the following local stability criterion.
\begin{itemize}
    \item If $r < f$, the only nonnegative spatially homogeneous steady state is $(0,f)$, which is locally asymptotically stable.
    \item If $r > f$, both spatially homogeneous steady states, $(0,f)$ and $(u_*,v_*)$, are biologically meaningful. In this case, $(0,f)$ is unstable, whereas $(u_*,v_*)$ is locally asymptotically stable.
\end{itemize}
The main objective of this work is to extend this local stability criterion to a global result, obtaining asymptotic convergence of solutions to these steady states under suitable conditions. 

\section{Global stability of spatially homogeneous steady states}
\subsection{The case $r<f$}
We begin by analyzing the case in which the external supply $f$ exceeds the logistic growth rate $r$. We recall that in this case, the only spatially homogeneous steady state is $(0,f)$, which is locally asymptotically stable. Here, we prove that if $r$ is large enough, then any solution stabilizes toward such steady state. To do so, we consider the following functionals
\begin{equation}\label{2.1}
    E_1(t) := \int_\Omega u + \frac{k}{2} \int_\Omega (v-f)^2, \quad F_1(t) := \int_\Omega u^2 + \int_\Omega (v-f)^2, \quad t >0,
\end{equation}
where $k>0$ will later be adjusted so that $E_1$ decreases along trajectories. This results in the following lemma.
\begin{lemma}\label{l2.1}
Let $f > r$ and assume that $(u,v)$ is a global bounded classical solution of \eqref{1.1} and that $r \geq a$. Then there exist two positive constants $k$ and $c_1$ such that $E_1$ and $F_1$ satisfy
$$E_1(t) \geq 0, \quad \frac{d}{dt}E_1(t) \leq - c_1 F_1(t) \leq 0, \quad \text{for all } t >0.$$
\end{lemma}

\begin{proof}

Firstly, its is evident from the nonnegativity of $u$ that for any $k > 0$, $E_1(t) \geq 0$ for all $t >0$. 

Next, we compute the time derivative of each of the summands in $E_1$. Integrating by parts we obtain
\begin{equation} \nonumber
    \begin{split}
        \frac{d}{dt}\int_\Omega u & =   \int_\Omega \Big (D \Delta u + \chi \nabla \cdot ( u \nabla v) + ru(1-u) - uv \Big) = r \int_\Omega u (1-u) - \int_\Omega uv \\
        & = - r\int_\Omega u^2 + \int_\Omega u(r-v) = - r\int_\Omega u^2 + \int_\Omega u(f-v) + \int_\Omega u(r-f) \\
        & < - r\int_\Omega u^2 + \int_\Omega u(f-v) , \quad \text{for all } t >0.
    \end{split}
\end{equation}
where we used that $f > r$ in the last step. Next, by Young's inequality, for an arbitrary $\varepsilon_1 > 0$ we have
\begin{equation}\label{2.10}
   \begin{split}
    \frac{d}{dt}\int_\Omega u &\leq - r\int_\Omega u^2 + \frac{1}{4 \varepsilon_1}  \int_\Omega u^2 + \varepsilon_1 \int_\Omega (v-f)^2= - \left(r - \frac{1}{4 \varepsilon_1} \right)\int_\Omega u^2 + \varepsilon_1 \int_\Omega (v-f)^2, 
\end{split}
\end{equation}
for all $t>0$. Similarly, for any given $k>0$
\begin{equation} \label{2.11}
    \begin{split}
    \frac{1}{2}\frac{d}{dt} \int_\Omega (v-f)^2  &=  \int_\Omega (v-f) ~v_t =  \int_\Omega (v-f) ( \Delta v + a u - v + f ) = - \int_\Omega |\nabla v|^2 + a \int_\Omega u(v-f)\\
     & - \int_\Omega (v-f)^2 \leq - \int_\Omega |\nabla v|^2 + \frac{a^2 k}{2r} \int_\Omega (v-f)^2 + \frac{r}{2k} \int_\Omega u^2 - \int_\Omega (v-f)^2 \\
     & = - \int_\Omega |\nabla v|^2 - \left(1- \frac{a^2 k}{2r} \right) \int_\Omega (v-f)^2 + \frac{r}{2k} \int_\Omega u^2, \quad \text{for all } t>0.
    \end{split}
\end{equation}
Thus, combining \eqref{2.10} and \eqref{2.11} one obtains
\begin{equation} \label{2.12}
\begin{split}
    \frac{d}{dt} E_1(t) &= \frac{d}{dt} \left(\int_\Omega u + \frac{k}{2} \int_\Omega (v-f)^2 \right) < -\left(\frac{r}{2} - \frac{1}{4 \varepsilon_1} \right)\int_\Omega u^2  - \left(k- \frac{a^2 k^2}{2r} - \varepsilon_1 \right) \int_\Omega (v-f)^2 
\\ & - k \int_\Omega |\nabla v|^2  \leq -\left(\frac{r}{2} - \frac{1}{4 \varepsilon_1} \right)\int_\Omega u^2  - \left(k- \frac{a^2 k^2}{2r} - \varepsilon_1 \right) \int_\Omega (v-f)^2, \quad \text{for all } t>0.    
\end{split}
\end{equation}
The last part is to adequately select a $k$ and $\varepsilon_1$ such that both coefficients are non-positive, this is $\displaystyle \frac{r}{2} - \frac{1}{4 \varepsilon_1} \geq 0$, and $\displaystyle k- \frac{a^2 k^2}{2r} - \varepsilon_1 \geq 0$.
It follows immediately that $k$ and $\varepsilon_1$ must fulfill
\begin{equation} \label{2.13}
    \frac{1}{2r} \leq \varepsilon_1 \leq k - \frac{a^2k^2}{2r} =: Q(k).
\end{equation}
Thus, for \eqref{2.13} to hold, the maximum of the parabola $Q(k)$ cannot to be attained below $1/2r$. 

If $a>0$, the vertex of the parabola is located at $\big(k^*, Q(k^*)\big) := \displaystyle\left( \frac{r}{a^2}, \frac{r}{2a^2} \right)$, and hence \eqref{2.13} holds if $Q(k^*) \geq 1/2r$, or equivalently if $r \geq a$. In this way, $c_1$ can be determined as the minimum of the coefficients in \eqref{2.13}, for any $k$ and $\varepsilon_1$ within the feasible region.

If $a=0$, any $r > 0$ satisfies \eqref{2.13} taking for instance $k = 1/r>0$ and $\varepsilon_1 = 3/4r>0$, taking $c_1$ again as the minimum of the coefficients in \eqref{2.13}.
\end{proof}

\begin{theorem}\label{tc1}
    Consider system \eqref{1.1} with a constant supply $f>r$ and  any choice of $(u_0,v_0)$ with $u_0 \in C^0(\Omega)$ and $v_0 \in W^{1,\theta}(\Omega)$, for some $\theta>n$. If moreover $r \geq a$, then $(u,v)$, the unique global solution to the system, is such that
$$\|u(\cdot,t)\|_{L^\infty(\Omega)} + \|v(\cdot,t)-f\|_{L^\infty(\Omega)} \to 0 \quad as ~t \to \infty.$$  
\end{theorem}
\begin{proof}
    First, by Lemma \ref{l2.1}, as $r \geq a$ we can find $k,~c_1 >0$ such that $E_1(t) \geq 0$ and $\displaystyle \frac{d}{dt}E_1(t) \leq - c_1 F_1(t) \leq 0$ for all $t >0$. Given the boundedness of $u$ and $v$, this implies that for any $t >1$
    \begin{equation} \label{2.14}
          \int_1^\infty F_1(s) ~ds \leq - \frac{1}{\varepsilon_1} \int_1^\infty \left( \frac{d}{ds}  E_1(s) \right) ds \leq \frac{1}{\varepsilon_1} (E_1(1) - E_1(t)) \leq \frac{E_1(1)}{\varepsilon_1} < \infty.
    \end{equation}
    Moreover, by standard parabolic theory \cite{Ladyzhenskaya1968}, there exists $\theta \in (0,1)$ and $c>0$ such that
    \begin{equation}\label{2-holder}
           \|u(\cdot,t)\|_{C^{2+\theta, 1+\frac{\theta}{2}}(\bar{\Omega}\times [t,t+1])} + \|v(\cdot,t)\|_{C^{2+\theta, 1+\frac{\theta}{2}}(\bar{\Omega}\times [t,t+1])} \leq c, \quad \text{for all } t \geq 1.
    \end{equation}   
    As a result, $F_1(t)$ is uniformly continuous in $(1, \infty)$ and since $ \int_1^\infty F_1(s)~ ds< \infty$, we have $F_1(t) \to 0$ as $t \to \infty$, this is 
    $$F_1(t) = \|u(\cdot,t)\|^2 _{L^2(\Omega)} + \|v(\cdot,t)-f\|^2_{L^2(\Omega)} \to 0 \quad \text{as} ~t \to \infty.$$
    In particular, as a consequence of \eqref{2-holder}, which yields global boundedness of $u$ and $v$ in $W^{1,\infty}(\Omega)$ for all $t>1$, through a Gagliardo-Niremberg type inequality, we can improve the convergence to the $L^\infty$ norm, namely
    \begin{equation*}
    \begin{split}           
    ||u||_{L^\infty (\Omega)}+ ||v-f||_{L^\infty (\Omega)} &\leq C\left( ||u||_{W^{1,\infty}(\Omega)}^{\frac{n}{n+2}} \cdot ||u||_{L^2(\Omega)}^{\frac{2}{n+2}} +  ||v-f||_{W^{1,\infty}(\Omega)}^{\frac{n}{n+2}} \cdot ||v-f||_{L^2(\Omega)}^{\frac{2}{n+2}}  \right) \\[1 ex]
    &\leq \overline{C} \left(||u||^{\frac{2}{n+2}}_{L^2(\Omega)} + ||v-f||^{\frac{2}{n+2}}_{L^2(\Omega)}  \right) \to 0, \quad \text{as } t \to \infty,
    \end{split}
    \end{equation*}
    for some $C, ~\overline{C} >0$, which finishes the proof.
    \qed
\end{proof}
\subsection{The case $r>f$}
Next, we proceed similarly with $r>f$. In this case, the state $(0,f)$ is now unstable, while $(u_*, v_*)$ is positive and locally asymptotically stable. This time, we make use of the following functionals, which were introduced in \cite{BaiWinkler2016}, and have been applied to a wide range of chemotaxis problems, for instance \cite{Mizukami2017, Zhang2019, Deng2023}.
\begin{equation} \label{2.15}
\begin{split}
    E_2(t) &:= \int_\Omega \left(u - u_* - u_* \ln \frac{u}{u_*} \right) + \frac{k}{2} \int_\Omega (v-v_*)^2, 
    \\F_2(t) &:=  \int_\Omega \left (\frac{| \nabla u|}{u} \right )^2 + \int_\Omega |\nabla v |^2 + \int_\Omega (u-u_*)^2 + \int_\Omega (v-v_*)^2,    
\end{split}
\end{equation}
In this case, a further largeness hypothesis is required for $r$, in order to prove the following counterpart to Lemma \ref{l2.1}. 
\begin{lemma} \label{l2.3}
    Let $r>f$ and assume that $(u,v)$ is a global bounded classical solution of \eqref{1.1}. Then there exists $r_c>0$ and two positive constants $k$ and $c_2$ such that if $r > r_c$, the functionals $E_2$ and $F_2$ given in \eqref{2.15} satisfy
$$E_2(t) \geq 0, \quad \frac{d}{dt}E_2(t) \leq - c_2 F_2(t) \leq 0, \quad \text{for all } t >0.$$
\end{lemma}
\begin{proof}
The non-negativeness of $E_2$ can be proven using a Taylor expansion as in the original article by \cite{BaiWinkler2016}. Next, computing the time derivative of the first term in $E_2$ we have
\begin{equation}\label{2.18}
    \begin{split}
        & \frac{d}{dt}  \int_\Omega \left(u - u_* - u_* \ln \frac{u}{u_*} \right) =  \int_\Omega u_t \left( 1- \frac{u_*}{u} \right) 
        \\
       & \hspace{-0.3 cm}= D \int_\Omega \Delta u \left( 1- \frac{u_*}{u} \right)+ \chi \int_\Omega \nabla \cdot (u \nabla v) \left( 1- \frac{u_*}{u} \right) + r \int_\Omega u(1-u) \left( 1- \frac{u_*}{u} \right) - \int_\Omega uv \left( 1- \frac{u_*}{u} \right),           
        \end{split}
    \end{equation}
for all $t>0$. Integrating by parts the first two terms in \eqref{2.18} yields
    \begin{equation} \label{2.19}
    \begin{split}
        &\int_\Omega \Delta u \left( 1- \frac{u_*}{u} \right) = - \int_\Omega \nabla u \cdot \nabla \left( 1- \frac{u_*}{u} \right) = u_*\int_\Omega  \nabla u \cdot \nabla \left( \frac{1}{u} \right) = - u_* \int_\Omega \left( \frac{|\nabla u |}{u}\right)^2, \\
         &\int_\Omega \nabla \cdot (u \nabla v) \left( 1- \frac{u_*}{u} \right) = - \int_\Omega u \nabla v \cdot \nabla \left( 1- \frac{u_*}{u} \right) = - u_* \int_\Omega \frac{1}{u} \nabla u \cdot \nabla v,
    \end{split}
    \end{equation}
for all $t >0$. With respect to the last two summands in \eqref{2.18}, a rearrangement of the terms leads to
\begin{equation} \label{2.21}
        \begin{split}
           & r \int_\Omega u(1-u) \left( 1- \frac{u_*}{u} \right) - \int_\Omega uv \left( 1- \frac{u_*}{u} \right)
            = \int_\Omega (u-u_*)(r- ru - v)  \\
           & = \int_\Omega (u-u_*)(r- ru - ru_* + ru_* - v + v_* - v_*) = \int_\Omega (u-u_*) \left[ -r(u-u_*) - (v-v_*)\right] \\
           & + \int_\Omega (u-u_*) (r-ru_* - v_*) = - r \int_\Omega (u-u_*)^2 - \int_\Omega (u-u_*)(v-v_*),  \quad \text{for all } t>0.    
        \end{split}
\end{equation}
where we used that $\displaystyle r-ru_* - v_* = r\left(1-u_* - \frac{v_*}{r}  \right) = 0$, as a consequence of $(u_*, v_*)$ solving the first equation of system \eqref{1.1}. Thus, by substituting \eqref{2.19}--\eqref{2.21} into \eqref{2.18} we obtain
\begin{equation} \label{2.22}
\begin{split}
    \frac{d}{dt} \int_\Omega \left(u - u_* - u_* \ln \frac{u}{u_*} \right) &= - Du_* \int_\Omega \left( \frac{|\nabla u |}{u}\right)^2 - \chi  u_* \int_\Omega \frac{\nabla u \cdot \nabla v}{u}
    \\& - r \int_\Omega (u-u_*)^2 - \int_\Omega (u-u_*)(v-v_*), \quad \text{for all } t>0.    
\end{split}
\end{equation}
Proceeding similarly with the second term in $E_2$ one obtains    \begin{equation} \label{2.23}
        \begin{split}
           & \frac{1}{2} \frac{d}{dt} \int_\Omega (v-v_*)^2 =  -\int_\Omega |\nabla v|^2 + a \int_\Omega (v-v_*)(u-u_*) - \int_\Omega (v-v_*)^2, \quad \text{for all } t>0.
        \end{split}
\end{equation}
Thus, combining \eqref{2.22} and \eqref{2.23} results in
\begin{equation}\nonumber
    \begin{split}
        \frac{d}{dt} E_2(t) &=  - Du_* \int_\Omega \left( \frac{|\nabla u |}{u}\right)^2 - k \int_\Omega |\nabla v|^2 - \chi  u_* \int_\Omega \frac{\nabla u \cdot \nabla v}{u} - r \int_\Omega (u-u_*)^2 \\
        &- k \int_\Omega (v-v_*)^2 - (1-ka) \int_\Omega (u-u_*)(v-v_*),
    \end{split}
\end{equation}
which can be rewritten in terms of two quadratic forms as
\begin{equation} \label{2.24}
    \frac{d}{dt} E_2(t) = - \int_\Omega X^t \cdot \vect{P} \cdot X - \int_\Omega Y^t \cdot \vect{S} \cdot Y,
\end{equation}
with $X := \big ( u - u_*, v-v_* \big)^t$, $Y := \left( \displaystyle\frac{|\nabla u|}{~u}, |\nabla v| \right)^t$, and
\begin{equation} \label{2.25}
    \vect{P} := \begin{pmatrix}
        r &  \displaystyle \frac{1-ka}{2} \\\\
        \displaystyle \frac{1-ka}{2} & k
    \end{pmatrix}, \quad  \vect{S}:= \begin{pmatrix}
        D u_* & \displaystyle \frac{\chi \cdot u_*}{2} \\\\
        \displaystyle \frac{\chi \cdot u_*}{2} & k
    \end{pmatrix}.
\end{equation}
The last step of the proof is to show that the matrices $\vect{P}$ and $\vect{S}$ are positive-definite, which implies the existence of $c_2 > 0$ such that for all $X$ and $Y$
$$ X^t \cdot \vect{P} \cdot X \geq c_2 |X|^2 \quad \text{and} \quad Y^t \cdot \vect{S} \cdot Y \geq c_2 |Y|^2.$$
Thus, to ensure the positive-definiteness of both matrices, we employ Sylvester's criterion. Both matrices being of size $2\times 2$, the principal minors are simply given by
\begin{equation} \nonumber
    \begin{split}
       & M_1(\vect{P}) = r >0, \quad \hspace{0.42 cm} M_2(\vect{P}) = \det(\vect{P}) = rk - \frac{(1-ka)^2}{4}, \\
        & M_1(\vect{S}) = Du_* >0, \quad M_2(\vect{S}) = \det(\vect{S}) = Dku_* - \frac{\chi^2 u_*^2}{4}.
    \end{split}
\end{equation}
This means that $k>0$ has to be selected so that $\displaystyle M_2(\vect{P})>0$ and $\displaystyle M_2(\vect{S})>0$, this is
\begin{equation}\label{cambio}
    rk - \frac{(1-ka)^2}{4} >0, \quad k > k_{\min}(r) := \frac{\chi^2 (r-f)}{4D(r+a)},
\end{equation}
where in $k_{\min}$ we have substituted $u_*$ by its value, given in \eqref{estados}. Assuming first that $a>0$, solving for $k$ in the first inequality leads to
\begin{equation}\label{cambio-12}
    k \in \left(\frac{(a+2r) - 2\sqrt{r^2+ar}}{a^2}, ~\frac{(a+2r) + 2\sqrt{r^2+ar}}{a^2} \right) =: \big(k_1(r), k_2(r)\big).
\end{equation}
Therefore, the feasible region for $(r,k)$ fulfilling \eqref{cambio} is determined by those $r>f$ such that $k_{\min}(r) < k_2(r)$. There are however two alternatives for this. As $k_{\min}(r)$ is monotonically increasing, with a finite limit $\lim_{r \to \infty} k_{\min} = \chi^2/4D$, whereas $\lim_{r \to \infty} k_2(r) = +\infty$, the graphs of $k_{\min}$ and $k_2(r)$ may not intersect, or have two different intersections (or one with multiplicity two), as schematically depicted in Figure \ref{fig:f_cort}.
\begin{figure}[htp]
    \centering
    \includegraphics[width=13.25 cm]{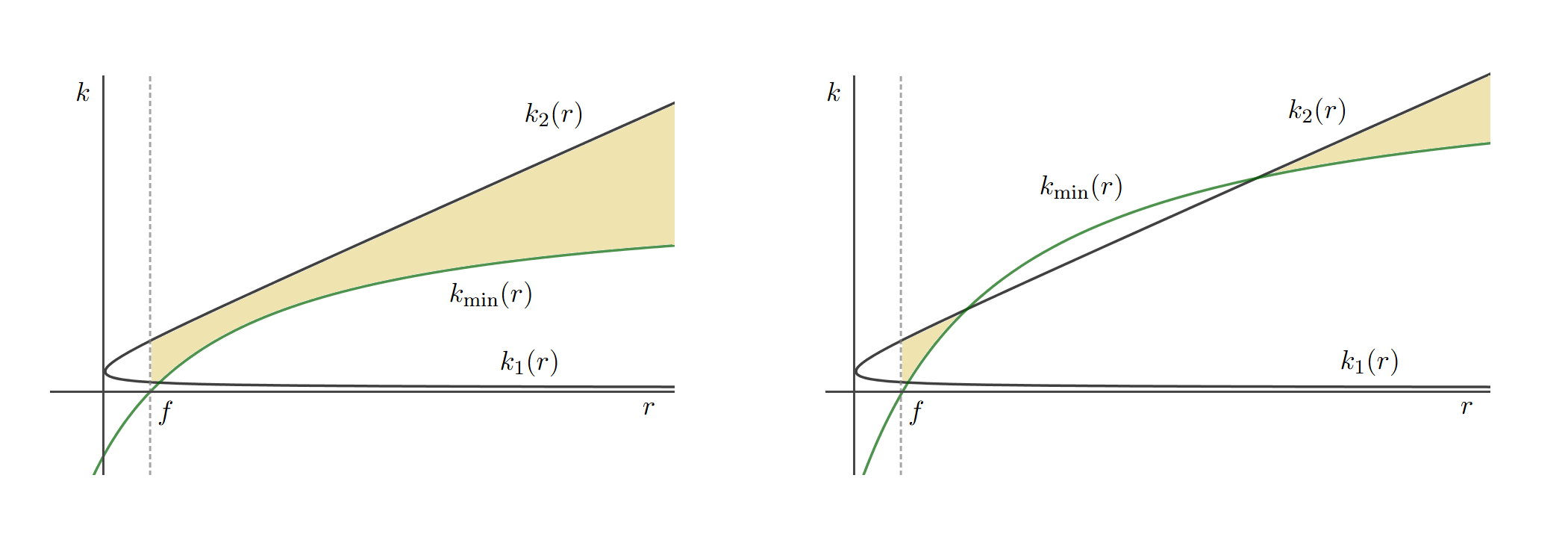}
    \caption{Two different possibilities for $k_{\min}(r)$ and $k_2(r)$ for \eqref{cambio} to hold. The feasible region is shaded.}
    \label{fig:f_cort}
\end{figure}\\
If, as shown on the left panel, $k_{\min}$ and $k_2$ do not intersect, then for any $r>f$ and an arbitrary $k \in (k_{\min}(r), k_2(r))$, condition \eqref{cambio} is satisfied. If, on the contrary, $k_{\min}$ and $k_2$ intersect twice, as represented on the right panel, \eqref{cambio} can only hold for those values of $r$ for which $k_{\min}(r)$ lies below $k_2(r)$, producing the feasible region shaded.

This is equivalently characterized by the following cubic polynomial $\mathcal{P}(r)$, obtained from computing the intersections between the quadratic in \eqref{cambio}, which gives rise to $k_1$ and $k_2$, and $k_{\min}$.
\begin{equation}\label{cambio2}
\begin{split}
    \mathcal{P}(r) &:=  \big(16Da^2\chi^2\big)r^3 + \big(24 D a^3 \chi^2-a^4 \chi^4  - 16 fDa^2 \chi^2 - 16D^2a^2\big) r^2 \\
    &+ \big( 2 f a^4 \chi^4 + 8Da^4\chi^2 - 24fDa^3\chi^2 - 32D^2a^3\big)r - \big(16D^2a^4 + 8Da^4\chi^2f + a^4\chi^4 f^2 \big).
\end{split}
\end{equation}
As the leading coefficient is $16Da^2\chi^2 > 0$, and $\mathcal{P}(0) = - \big(16D^2a^4 + 8Da^4\chi^2f + a^4\chi^4 f^2 \big)<0$, then $\mathcal{P}(r)$ must have at least one positive root. Denoting by $r_1$ the first positive root of $\mathcal{P}$, it corresponds to the intersection of $k_1$ and $k_{\min}$. If there are no further positive roots aside from $r_1$, then $k_{\min}$ and $k_2$ do not intersect, and thus we can take $r_c= r_1$. 

On the contrary, if $r_1$ is not the only positive real root, by Descartes' Rule of Signs all three roots of $\mathcal{P}$ must be positive, this is $r_1<r_2 \leq r_3$, corresponding to the situation on the right panel of Figure \ref{fig:f_cort}. In this case, \eqref{cambio} is only satisfied if either $r \in (f,r_1)$ or $r \in (r_3, +\infty)$, and thus we define $r_c = r_3$. In this way, on both cases, for $r > r_c$, matrices $\vect{P}$ and $\vect{S}$ are positive definite, concluding the result. 

In the case $a=0$, \eqref{cambio} directly reduces to
$rk > 0$ and  $k>\frac{\chi^2  (r-f)}{4D}$. Hence, any $r>f$ satisfies the inequality by taking an arbitrary $k$ above $\frac{\chi^2 (r-f)}{4D}$. 

The positive-definiteness of $\vect{P}$ and $\vect{S}$ combined with \eqref{2.24} directly leads to the result, given the structure of $F_2$, $X$ and $Y$.
\qed
\end{proof}
Lastly, the following theorem ensures the counterpart to Theorem \ref{tc1}, the convergence of solutions to $(u_*, v_*)$ when $r > f$ and moreover $r > r_c$. The only difference is that in this case, the existence of the unstable steady state $(0,f)$ implies that any solution converges to $(u_*, v_*)$, except if the initial data is precisely $(0,f)$. For brevity reasons, we omit the proof, as it follows the same steps. 
\begin{theorem} \label{tc2}
Consider system \eqref{1.1} with a constant supply $f<r$ and any choice of $(u_0,v_0) \not\equiv (0,f)$ with $u_0 \in C^0(\Omega)$ and $v_0 \in W^{1,\theta}(\Omega)$, for some $\theta>n$. Then, if $r > r_c$, the solution $(u,v)$ satisfies
$$\|u(\cdot,t)-u_*\|_{L^\infty(\Omega)} + \|v(\cdot,t)-v_*\|_{L^\infty(\Omega)} \to 0 \quad as ~t \to \infty.$$
\end{theorem}
\section{Conclusion}
Throughout this work, we have proved Theorems \ref{tc1} and \ref{tc2} concerning global asymptotic stability of the spatially homogeneous steady states $(0,f)$ and $(u_*, v_*)$, ensuring the convergence of solutions in $L^\infty(\Omega)$ norm. These were however obtained under additional the requirements that $r \geq a$ in the first case, and $r > r_c$ in the second one.

Although $r_c$ is not explicitly defined, as it involves solving the cubic polynomial \eqref{cambio2}, some particular cases can be explored. As already mentioned in the proof of Lemma \ref{l2.3}, if $a=0$, then $r_c =0$. In this case, where the chemorepellent is not self-produced by the species, the second equation in \eqref{1.1} is linear and uncoupled, and solutions can easily be seen to stabilize toward the steady state $v\equiv f$. Thus, once diffusion homogenizes the chemical concentration over $\Omega$, chemotaxis plays no role, and the bacterial density also equilibrates toward its corresponding steady state.  Another important case is the taxis-free scenario, i.e. $\chi = 0$. In such case, the conditions in \eqref{cambio} are also fulfilled by any $r>0$ if $k$ is taken between $k_1$ and $k_2$, given in \eqref{cambio-12}.
 
In the regime not covered by these parameter restrictions, different phenomena may arise. In \cite{FH1}, numerical evidence shows pattern formation arising after initiating the system with a perturbation of the nontrivial steady state $(u_*,v_*)$, considering a large enough values of $\chi$. For this to occur, $r$ has to be moderately high --- in particular between $r_2$ and $r_3$, the two largest roots of the polynomial $\mathcal{P}$ in \eqref{cambio2}--- so as to counteract the effect of the chemorepellent, while not large enough to drive the system to the spatially homogeneous equilibrium.

\section*{Acknowledgments}
This work was supported by Project PID2022-141114NB-I00 from the Spanish Ministry of Science and Innovation and by Grant FPU23/03170 from the Spanish Ministry of Science, Innovation and Universities (F.H.-H.)

\end{document}